 \documentclass[11pt]{article}

\usepackage{graphicx}

\usepackage{palatino}
\usepackage[mathcal]{euler}

\usepackage{amsmath,amsthm,amsfonts,amssymb,latexsym,amscd,enumerate}
\usepackage{xy}
\xyoption{all}

\theoremstyle{plain}
    \newtheorem{theorem}{Theorem}[section]
    \newtheorem{lemma}[theorem]{Lemma} 
    \newtheorem{corollary}[theorem]{Corollary}
    \newtheorem{proposition}[theorem]{Proposition}
    
    \newtheorem{assumption}[theorem]{Assumption}
    \theoremstyle{definition}
    \newtheorem{definition}[theorem]{Definition}
    
    \newtheorem{example}[theorem]{Example}

    \newtheorem{remark}[theorem]{Remark}

\theoremstyle{remark}

\numberwithin{equation}{section}

    \newcommand{\C}{\mathbb{C}} 
    
    \newcommand{\Z}{\mathbb{Z}}

        \newcommand{\kg}{\mathfrak{g}} 
\newcommand{\kh}{\mathfrak{h}} 
\newcommand{\kp}{\mathfrak{p}}
\newcommand{\kt}{\mathfrak{t}}
\newcommand{\kk}{\mathfrak{k}}
\newcommand{\XX}{\mathfrak{X}}

\newcommand{\cS}{\mathcal{S}}

\newcommand{\cH}{\mathcal{H}}
\newcommand{\cL}{\mathcal{L}}
\newcommand{\cF}{\mathcal{F}}

\newcommand{\cB}{\mathcal{B}}

\newcommand{\cSM}{\cS}

\newcommand{\DM}{D}
\newcommand{\LM}{L}
\newcommand{\vM}{v}

 \newcommand{\QcwR}{$[Q^{\Spinc},R]=0$ }

    \DeclareMathOperator{\Ad}{Ad}
             
         \DeclareMathOperator{\SO}{SO}
     \DeclareMathOperator{\Spin}{Spin}
    \DeclareMathOperator{\Spinc}{\Spin^c}
        
                \DeclareMathOperator{\End}{End}
     \DeclareMathOperator{\ind}{index}
     \DeclareMathOperator{\Tor}{Tor}
    \DeclareMathOperator{\ds}{ds}
      \DeclareMathOperator{\DInd}{D-Ind}
   
\DeclareMathOperator{\mat}{mat}
\DeclareMathOperator{\Hom}{Hom}

\DeclareMathOperator{\rank}{rank}
\DeclareMathOperator{\reg}{reg}

\newcommand{\beq}[1]{\begin{equation} \label{#1}}
\newcommand{\eeq}{\end{equation}}

\title{An equivariant index for proper actions III: the invariant and discrete series indices}

\author{Peter Hochs\footnote{University of Adelaide, \texttt{peter.hochs@adelaide.edu.au}}\hspace{1.5mm} and Yanli Song\footnote{University of Toronto, \texttt{songyanl@math.utoronto.ca}}}

\date{\today}

\begin{document}

\maketitle

\begin{abstract}
We study two special cases of the equivariant index defined in part I of this series. We apply this index to deformations of $\Spinc$-Dirac operators, invariant under actions by possibly noncompact groups, with possibly noncompact orbit spaces. One special case is an index defined in terms of multiplicities of discrete series representations of semisimple groups, where we assume the Riemannian metric to have a certain product form. The other is an index defined in terms of sections invariant under a group action. We obtain a relation with the analytic assembly map, 
 quantisation commutes with reduction results, and Atiyah--Hirzebruch type vanishing theorems. The arguments are based on an explicit decomposition of $\Spinc$-Dirac operators with respect to a global slice for the action.
\end{abstract}

\tableofcontents

\section{Introduction}

In part I of this series \cite{HS}, an equivariant index was defined for actions by possibly noncompact groups, with possibly noncompact orbit spaces. It was shown to apply to natural deformations of Dirac-type operators as in \cite{BravermanK, Braverman, HS2, HS3}, which have been used successfully in geometric quantisation \cite{HM2, HS, TZ}. 

In this paper, we consider such deformations of $\Spinc$-Dirac operators. 
For actions by semisimple Lie groups with discrete series, on manifolds with 
 Riemannian metrics of a certain product form, the equivariant index can be expressed in terms of multiplicities of discrete series representations in the kernel of such an operator. This motivates the definition of the \emph{discrete series index}.
We also show that, for these deformed $\Spinc$-Dirac operators, the equivariant index generalises the \emph{invariant index} studied in \cite{Braverman, HM1}. The latter is defined in terms of sections invariant under a group action.

The assumption on the Riemannian metric means direct geometric arguments can be used to obtain these relations with the equivariant index. (For the invariant index, this assumption is not necessary because the claim can always be reduced to metrics of that form.) In the cocompact case, the discrete series index is also directly related to the analytic assembly map
 \cite{BCH, Kasparov}. Furthermore, we obtain \emph{quantisation commutes with reduction} results for the discrete series index and the invariant index. 
The result for the invariant index sharpens an asymptotic result in \cite{HM2}. In the cocompact case, this result reduces to a $\Spinc$-version of Landsman's conjecture \cite{HL, Landsman}. Finally, Atiyah and Hirzebruch's vanishing result \cite{AH} on compact $\Spin$ manifolds generalises to the discrete series index and the invariant index in a way analogous to the $K$-theoretic result in \cite{HMAH}. (The result for the invariant index is actually a special case of the result in \cite{HMAH}.)

As said above, the results for the discrete series index hold for Riemannian metrics of a certain form. It is an interesting question to what extent they generalise to arbitrary (complete, invariant) metrics, as is the case for the invariant index.

\subsection*{Overview}

In Section \ref{sec prelim}, we recall the definitions of the equivariant index of \cite{HS} and the invariant index of \cite{Braverman, HM1}, introduce the discrete series index, and state the main results.
 In Section \ref{sec decomp Dirac} we give a decomposition of $\Spinc$-Dirac operators in terms of a global slice for the action. This leads to an induction result, 
 Proposition \ref{prop Dirac M N 1}. Then, in 
Section \ref{sec induction}, we include $L^2$-inner products to decompose the kernels of Dirac operators on the relevant spaces. This allows us to prove the main results in 
Subsection \ref{sec proofs}.

\subsection*{Acknowledgements}

The first author was supported by the European Union, through Marie Curie fellowship PIOF-GA-2011-299300.

\section{Preliminaries and results}\label{sec prelim}

Throughout this paper, we consider a complete Riemannian manifold $M$. We will identify $T^*M \cong TM$ via the Riemannian metric where convenient. Furthermore, $G$ will be a Lie group, with a maximal compact subgroup $K$. Unless stated otherwise, we assume that $G$ is unimodular, with a bi-invariant Haar measure $dg$, and that $G/K$ is even-dimensional. (Some of the constructions in this paper apply to more general groups, but we only apply them under these assumptions here.) 
We suppose $G$ acts properly and isometrically on $M$. Let $\cS = \cS^+ \oplus \cS^- \to M$ be a $\Z_2$-graded, $G$-equivariant Hermitian vector bundle.
In most of this paper, $\cS$ will be the spinor bundle of an equivariant $\Spinc$-structure.
For any (odd) operator $D$ on sections of $\cS$, we write $D^{\pm}$ for the restriction of $D$ to sections of  $\cS^{\pm}$.

The proofs of the results stated in this section are given in Subsection \ref{sec proofs}.

\subsection{Product metrics}

The results about the discrete series index in this paper hold for Riemannian metrics on $TM$ of a certain form. 
By Abels' theorem \cite{Abels}, there is a $K$-invariant submanifold $N \subset M$ such that the action map defines a $G$-equivariant diffeomorphism
\beq{eq Abels}
G\times_K N \cong M.
\eeq
Here the left hand side if the quotient of $G\times N$ by the action by $K$ given by $k\cdot (g, n) = (gk^{-1}, kn)$, for $k \in K$, $g \in G$ and $n \in N$.

Let $\kp \subset \kg$ be a $K$-invariant subspace such that $\kg = \kk \oplus \kp$. Then under \eqref{eq Abels}, we have
\beq{eq Abels TM}
TM = G\times_K (TN \oplus (N \times \kp)) \to M.
\eeq
\begin{definition}
A \emph{product metric} on $TM$ is a $G$-invariant Riemannian metric induced by a $K$-invariant Riemannian metric on $TN$ and a $K$-invariant inner product on $\kp$ via the isomorphism \eqref{eq Abels TM}.
\end{definition}

In the results about discrete series representations, the Riemannian metric on $TM$ will be assumed to be a product metric. This will be indicated in the relevant places.

\subsection{A realisation of discrete series representations} \label{sec AS Parth}

The explicit realisation of discrete series representations by Atiyah and Schmid \cite{AS} and Parthasarathy \cite{Parthasarathy} plays an important role in this paper. This realisation involves Dirac operators on $G/K$.

Let $G$ be any Lie group. Fix a $K$-invariant inner product on $\kp$.
 Let $\pi_{\kp}$ be the standard representation of $\Spin(\kp)$. Since $G/K$ is even-dimensional,  $\pi_{\kp}$ splits as $\pi_{\kp} = \pi_{\kp}^+ \oplus \pi_{\kp}^-$. 
 Let $D_{G/K}$ be the operator
\begin{equation} \label{eq def DGK}
D_{G/K} = \sum_{j = 1}^k X_j \otimes c(X_j)
\end{equation}
on $C^{\infty}(G) \otimes \pi_{\kp}$, where $\{X_1, \ldots, X_k\}$ is an orthonormal basis of $\kp$. Here $c\colon \kp\to \End(\pi_{\kp})$ is the Clifford action. 

Throughout this paper, we assume that the adjoint representation
 \[
 \Ad\colon K\to \SO(\kp)
 \]
 lifts to $\Spin(\kp)$, i.e.\ that $G/K$ is $G$-equivariantly $\Spin$. This is true for a double cover of $G$. Via this lift, we view $\pi_{\kp}$ as a representation of $K$. We will write $\pi_{\kp}$ for the formal difference
 \[
 \pi_{\kp} := \pi^+_{\kp} - \pi^{-}_{\kp} \quad \in R(K).
 \]
 Consider the diagonal representation of $K$ in $C^{\infty}(G) \otimes \pi_{\kp}$. The space \mbox{$\bigl( C^{\infty}(G) \otimes \pi_{\kp}\bigr)^K$} is the space of smooth sections of the spinor bundle $G\times_K \pi_{\kp} \to G/K$.

Now suppose $G$ is connected and semisimple with discrete series, i.e.\ $\rank(G) = \rank(K)$. Then by  Proposition 1.1 in \cite{Parthasarathy}, the restriction of $D_{G/K}$ to $\bigl( C^{\infty}(G) \otimes \pi_{\kp}\bigr)^K$ is the $\Spin$-Dirac operator on $G/K$.
Let $T<K$ be a maximal torus, with Lie algebra $\kt \subset \kk$. Let $\kt^*_+ \subset \kt^*$ be a choice of (closed) positive Weyl chamber. Let $R$ be the set of roots of $(\kk_{\C}, \kt_{\C})$, and let $R^+$ be the set of positive roots with respect to $\kt^*_+$. We denote half the sum of these positive roots by $\rho_K$.
Let $\Lambda_+ \subset i\kt^*$ be the set of dominant integral weights with respect to $\kt^*_+$. 
In the $\Spinc$-setting, it is natural to parametrise the irreducible representations of $K$ by their infinitesimal characters, rather than by their highest weights. For $\lambda \in \Lambda_+ + \rho_K$, let $\pi^K_{\lambda}$ be the irreducible representation of $K$ with infinitesimal character $\lambda$, i.e.\ with  highest weight $\lambda - \rho_K$. 

The discrete series of $G$ was realised in Theorem 9.3 in \cite{AS} and Theorem 1 in \cite{Parthasarathy}.
\begin{theorem}[Atiyah--Schmid, Parthasarathy] \label{thm AS Parth}
Let $\lambda \in \Lambda_+ + \rho_K$. One has
\[
\bigl( \pi^K_{\lambda}  \otimes \ker_{L^2} (D_{G/K}^-) \bigr)^K = 0.
\]
If $\lambda$ is singular, then also
\[
\bigl( \pi^K_{\lambda} \otimes \ker_{L^2} (D_{G/K}^+) \bigr)^K = 0.
\]
If $\lambda$ is regular, then 
\[
\bigl( \pi^K_{\lambda}  \otimes \ker_{L^2} (D_{G/K}^+) \bigr)^K = \pi_{\lambda}^{\ds},
\]
where $\pi_{\lambda}^{\ds}$ is
the discrete series representation of $G$ with Harish--Chandra parameter $\lambda$.
\end{theorem}

\subsection{$\Spinc$-Dirac operators and the equivariant index} \label{sec def quant}

We return to case where $G$ is any Lie group. Now suppose that $M$ is even-dimensional, and that it has a $G$-equivariant $\Spinc$-structure. The vector bundle $\cS$ is taken to be the spinor bundle associated to the $\Spinc$-structure. We denote the determinant line bundle of the $\Spinc$-structure by $L \to M$, and choose a $G$-invariant Hermitian connection $\nabla^L$ on $L$. Together with the Levi--Civita connection on $TM$, this induces a connection $\nabla^{\cS}$ on $\cS$ (see e.g.\ Proposition D.11 in \cite{LM}). This in turn defines a $\Spinc$-Dirac operator $D$ on $\cS$, by
\[
D\colon \Gamma^{\infty}(\cS) \xrightarrow{\nabla^{\cS}} \Omega^1(M; \cS) \xrightarrow{c} \Gamma^{\infty}(\cS).
\]
Here $c$ denotes the Clifford action by $TM \cong T^*M$ on $\cS$.



Let
\[
\mu\colon M \to \kg^*
\]
be the \emph{$\Spinc$-momentum map}, defined by
\begin{equation} \label{eq def mu}
2i  \mu_X = \cL^L_X - \nabla^L_{X^M} \quad \in \End(L) = C^{\infty}(M, \C).
\end{equation}
Here $\mu_X$ is the pairing of $\mu$ with an element $X \in \kg$, $\cL^L$ denotes the Lie derivative of sections of $L$, and $X^M$ is the vector field induced by $X$. By Lemma 5.3 in \cite{HM2}, the connection $\nabla^L$ can be chosen so that $\mu(N) \subset \kk^*$. (We identify $\kk^*$ with the annihilator of $\kp$ in $\kg^*$.) We assume $\nabla^L$ was chosen in this way.

In addition, fix a $K$-invariant inner product $(\relbar, \relbar)^{\kg}$ on $\kg$ extending the one on $\kp$, such that $\kk \perp \kp$.
Consider the metric $\{(\relbar, \relbar)_m\}_{m \in M}$ on the trivial vector bundle $M \times \kg \to M$
defined by
\[
\bigl(X,Y \bigr)_{gn} := (\Ad(g)^{-1}X, \Ad(g)^{-1}Y)^{\kg}
\]
for $X, Y \in \kg$, $g\in G$ and $n\in N$. Let $\mu^*\colon M \to \kg$ be the map defined by
\[
\langle \mu(m), X \rangle = (X, \mu^*(m))_m,
\]
for $X \in \kg$ and $m \in M$. Consider the $G$-invariant vector field $v$ on $M$ defined by
\begin{equation} \label{eq v mu}
v_m = 2\bigl(\mu^*(m)\bigr)^M_m,
\end{equation}
where $m \in M$, $\bigl(\mu^*(m)\bigr)^M$ is the vector field induced by $\mu^*(m) \in \kg$, and the factor 2 was included for consistency with \cite{HM1, HM2, TZ}. For a real-valued function $f\in C^{\infty}(M)^G$, the \emph{Dirac operator deformed by $fv$} is the operator
\begin{equation} \label{eq def Dv}
D_{fv} := D + ic(fv)
\end{equation}
on smooth sections of $\cS$.
 Already in the compact case, such a deformation was used by Tian and Zhang \cite{TZ} to prove Guillemin and Sternberg's quantisation commutes with reduction conjecture.
\begin{assumption} \label{ass taming}
We assume that the zeroes of $v$ form a \emph{cocompact} subset of $M$.
\end{assumption}



In \cite{HS2}, an equivariant index was defined for proper actions by possibly noncompact groups, with possibly noncompact orbit spaces. It was shown that this index applies to deformed Dirac operators (of a more general kind than the ones studied here). 
For any nonnegative function $\psi \in C^{\infty}(M)^G$, a nonnegative function $f \in C^{\infty}(M)^G$ is called \emph{$\psi$-admissible} if, outside a cocompact subset of $M$, we have
\[
\frac{f}{\|df\| + f + 1} \geq \psi.
\]
By Theorem 3.12 in \cite{HS2}, there is a nonnegative function $\psi \in C^{\infty}(M)^G$ such that for all $\psi$-admissible functions $f \in C^{\infty}(M)^G$, we have a well-defined equivariant index
\beq{eq def G index}
\ind_G(D_{fv}) := \Bigl[L^2(\cS), \frac{D_{fv}}{\sqrt{D_{fv}^2 + 1}}, \pi_{G, G/K} \Bigr] \in KK(C_0(G/K)\rtimes G, \C).
\eeq
Here $C_0(G/K)\rtimes G$ is a crossed-product $C^*$-algebra \cite{Williams}, and
the $*$-representation $\pi_{G, G/K}\colon C_0(G/K)\rtimes G \to \cB(L^2(\cS))$ is given by
\[
\bigl(\pi_{G, G/K} (\varphi)s\bigr)(gn) = \int_G \varphi(g', gK) g'\cdot (s(g'^{-1}gn))\, dg,
\]
for $\varphi \in C_c(G, C_0(G/K))$, $s \in L^2(\cS)$, $g \in G$ and $n\in N$. Via the Morita equivalence $C_0(G/K)\rtimes G \sim C^*K$, this index can be identified with an element of $KK(C^*K, \C) = \hat R(K)$. 

In the special case where $G=K$ is compact, the index \eqref{eq def G index} reduces to the index 
\beq{eq index K}
\ind_K(D_{fv})\in \hat R(K)
\eeq
studied by Braverman in \cite{BravermanK}.

\subsection{The invariant index}

In \cite{Braverman, HM1, HM2}, an index for proper actions is studied, defined with respect to sections invariant under the group action. 

A section of a vector bundle invariant under a proper action by a noncompact group cannot be square-integrable. For that reason, we use a  different Hilbert space of invariant sections.
Let $h\in C^{\infty}(M)$ be a \emph{cutoff function}, which means that it has compact support on $G$-orbits, and satisfies
\[
\int_G h(gm)^2 \, dg = 1
\]
for all $m\in M$. Then a section $s$ of $\cS$ is called \emph{transversally $L^2$} if $hs$ is $L^2$. This condition is independent of the cutoff function $h$ if $s$ is $G$-invariant. It was shown (for more general Dirac-type operators) in \cite{Braverman, HM1} that there is a nonnegative function $\psi \in C^{\infty}(M)^G$ such that for all $\psi$-admissible $f\in C^{\infty}(M)^G$, the spaces
\[
\ker_{L^2_T}(D_{fv}^{\pm})^G := \{s\in \Gamma^{\infty}(\cS^{\pm})^G;  \text{ $s$ transversally $L^2$ and $D_{fv} s = 0$}\}
\]
are finite-dimensional. For such $f$, the \emph{invariant index}
\beq{eq def invar index}
\ind_{L^2_T}^G (D_{fv}) := \dim \bigl( \ker_{L^2_T}(D_{fv}^{+})^G  \bigr) - \dim \bigl(\ker_{L^2_T}(D_{fv}^{-})^G\bigr)
\eeq
is independent of $f$. 
It is also independent of the Riemannian metric, as long as $M$ is complete. 
%

It was conjectured in Remark 4.4 in \cite{HS3} that, for more general Dirac-type operators, the invariant index can be recovered from the equivariant index of \eqref{eq def G index} as in the following result for $\Spinc$-Dirac operators.
\begin{proposition}\label{prop invar index}
If $f$ is $\psi$-admissible for $\psi$ as in the definitions of the indices \eqref{eq def G index} and \eqref{eq def invar index}, then for any $G$-invariant, complete Riemannian metric on $M$,
\[
\ind_{L^2_T}^G (D_{fv}) = \dim\bigl(\ind_{G} (D_{fv})^K \bigr),
\]
where on the right hand side, we view $\ind_{G} (D_{fv})$ as an element of $\hat R(K)$.
\end{proposition}

\subsection{The discrete series index} \label{sec ds index}

For now, let $\cS$ be any $\Z_2$-graded, Hermitian, $G$-equivariant vector bundle, and let $D$ be any odd, self-adjoint,  $G$-equivariant operator on $L^2(\cS)$.
We  assume that $G$ is connected and semisimple with discrete series. 
Let $\hat G_{\ds} \subset \hat G$ be the discrete part of the unitary dual of $G$. 
\begin{definition}
%
The \emph{discrete series representation group} of $G$ is the Abelian group
\[
R_{\ds}(G) := \Bigl\{ \bigoplus_{\pi \in \hat G_{\ds}} m_{\pi} \pi; m_{\pi} \in \Z, \text{nonzero for finitely many $\pi$}\Bigr\}.
\]
The \emph{completed discrete series representation group} of $G$ is the Abelian group
\[
\hat R_{\ds}(G) := \Bigl\{ \bigoplus_{\pi \in \hat G_{\ds}} m_{\pi} \pi; m_{\pi} \in \Z\Bigr\} \cong \Hom_{\Z}(R_{\ds}(G), \Z).
\]
\end{definition}
If $G=K$ is compact then we have $R_{\ds}(K) = R(K)$ and $\hat R_{\ds}(K) = \hat R(K)$, the usual representation ring and its completion. 
\begin{definition}
If the multiplicity \mbox{$[\ker_{L^2}(D):\pi]$} of any $\pi \in \hat G_{\ds}$ in the kernel of $D$ is finite, then
 $D$ is called \emph{$\ds$-Fredholm}, and its \emph{discrete series index} is
\[
\ind_{\ds}(D) := \bigoplus_{\pi \in \hat G_{\ds}} \bigl([\ker_{L^2}(D^+):\pi] - [\ker_{L^2}(D^-):\pi] \bigr) \pi \quad \in \hat R_{\ds}(G).
\]
\end{definition}

\begin{example}
If $M = G/H$, for a compact subgroup $H<G$, and $\cS$ is a vector bundle associated to a finite-dimensional representation of $H$, then Theorem 6.1 in \cite{CM} states that any elliptic pseudo-differential operator $D$ is $\ds$-Fredholm. In fact, one has
\[
\ker_{L^2}(D) \in R_{\ds}(G)
\]
for such operators. This was generalised to a larger class of groups in Theorem 6.2 in \cite{CM}. See  also Proposition 7.3.A.\ in the same paper, for Dirac operators.
\end{example}

From now on, $D$ will be a $\Spinc$-Dirac operator as in Subsection \ref{sec def quant}.
For a real-valued function $f\in C^{\infty}(M)^G$, let $D_{fv}$ be the deformed $\Spinc$-Dirac operator as in \eqref{eq def Dv}.
\begin{proposition} \label{prop def Dirac ds Fredholm}
Suppose the Riemannian metric on $TM$ is a product metric. Then
there is a nonnegative function $\psi \in C^{\infty}(M)^G$ such that for all $\psi$-admissible functions $f$, the operator $D_{fv}$ is $\ds$-Fredholm. Its $\ds$-index is independent of $f$ and $\nabla^L$. Furthermore, the $L^2$-kernel of $D_{fv}$ decomposes completely into discrete series representations.
\end{proposition}

\subsection{The discrete series index and other indices}

Let $G$ be connected and semisimple with discrete series.
Consider the \emph{Dirac induction} map
\beq{eq def Dirac ind}
\widehat{\DInd}_K^G\colon \hat R(K) \to \hat R_{\ds}(G)
\eeq
given by
\begin{equation} \label{eq def DInd}
\widehat{\DInd}_K^G(\pi^K_{\lambda}) = \bigl( \pi^K_{\lambda}  \otimes \ker_{L^2} (D_{G/K}^+) \bigr)^K - \bigl( \pi^K_{\lambda}  \otimes \ker_{L^2} (D_{G/K}^-) \bigr)^K.
\end{equation}
By Theorem \ref{thm AS Parth}, this map indeed takes values in $\hat R_{\ds}(G)$, it is surjective, and its restriction to the part of $\hat R(K)$ spanned by representations with regular infinitesimal characters is an isomorphism of Abelian groups. Also, the second term on the right hand side of \eqref{eq def DInd} is zero, but it was included for symmetry purposes.

In this subsection and the next, we suppose that the Riemannian metric on $TM$ is a product metric.
\begin{proposition} \label{prop G ds ind}
Suppose $f\in C^{\infty}(M)^G$ is $\psi$-admissible for $\psi$ both as in Proposition \ref{prop def Dirac ds Fredholm} and as in the definition of the index \eqref{eq def G index}. Then
\[
\ind_G(D_{fv}) = \pi_{\kp} \otimes (\widehat{\DInd}_K^G)^{-1} (\ind_{\ds} (D_{fv})) \quad \in \hat R(K).
\]
\end{proposition}
Since $G$ has discrete series representations, tensoring with $\pi_{\kp}$ is an invertible operation (see Lemma 4.7 in \cite{HS3}). So the equivariant index \eqref{eq def G index} determines the discrete series index of $D_{fv}$. 

Next, suppose that $M/G$ is compact. Then we have the \emph{analytic assembly map} \cite{Kasparov} from the Baum--Connes conjecture \cite{BCH}
\[
\mu_M^G\colon K_0^G(M) \to K_0(C^*_rG).
\]
Here $K_0^G(M)$ is the (even) $G$-equivariant \emph{$K$-homology} of $M$, and $K_0(C^*_rG)$ is the (even) \emph{$K$-theory} of the \emph{reduced goup $C^*$-algebra} of $G$.
Consider the inclusion map
\[
j\colon R_{\ds}(G) \hookrightarrow K_0(C^*_rG)
\]
given by $j(\pi) =  [d_{\pi}c_{\pi}]$, where $d_{\pi}$ is the formal degree of $\pi \in \hat G_{\ds}$, and $c_{\pi}$ is the matrix coefficient of any unit vector in the representation space of $\pi$ (see \cite{Lafforgue}). Let $p_{\ds}\colon K_0(C^*_rG) \to K_0(C^*_rG)$ be the projection onto the image of $j$.
\begin{proposition}\label{prop ass map}
If $M/G$ is compact, then the $\Spinc$-Dirac operator $D$ is $\ds$-Fredholm, its $\ds$-index lies in $R_{\ds}(G)$, and we have
\[
j\bigl(\ind_{\ds}(D) \bigr) = (-1)^{\dim(G/K)/2} p_{\ds}(\mu_M^G[D]).
\]
\end{proposition}

\subsection{$\Spinc$-quantisation commutes with reduction for the discrete series}\label{sec [Q,R]=0}

The \emph{quantisation commutes with reduction} principle of Guillemin and Sternberg \cite{GS} was extended from symplectic to $\Spinc$-manifolds by Paradan and Vergne \cite{PV1, PVI, PVII}. Their result for compact manifolds was generalised to noncompact ones by the authors of this paper \cite{HS}. The latter result generalises to discrete series representations.

 Let $\cF$ be the set of relative interiors of faces of the positive Weyl chamber $\kt^*_+$. 
For $\sigma \in \cF$, let $\kk_{\sigma}$ be the infinitesimal stabiliser of a point in $\sigma$. Let $R_{\sigma}$ be the set of roots of $\bigl( (\kk_{\sigma})_{\C}, \kt_{\C}\bigr)$, and let $R_{\sigma}^+ := R_{\sigma} \cap R^+$. Set
\[
\rho_{\sigma} := \frac{1}{2}\sum_{\alpha \in R^+_{\sigma}} \alpha.
\]
Note that if $\sigma$ is the interior of $\kt^*_+$, then $\rho_{\sigma} = 0$.

For any subalgebra $\kh \subset \kk$, let $(\kh)$ be its conjugacy class. Set
\[
\cH_{\kk} := \{ (\kk_{\xi}); \xi \in \kk\}.
\]
For $(\kh) \in \cH_{\kk}$, write
\[
\cF(\kh) := \{\sigma \in \cF; (\kk_{\sigma}) = (\kh)\}.
\]
Let $(\kk^M)$ be the conjugacy class (with respect to $K$) of the generic (i.e.\ minimal) infinitesimal stabiliser $\kk^M$ of the action by $K$ on $M$. 

For $i\xi \in i\kk^*$, consider the \emph{reduced space}
\[
 M_{i\xi} := \mu^{-1}(\xi)/G_{\xi}. 
\]
Here $G_{\xi}$ is the stabiliser of $\xi$ with respect to the coadjoint action. (Recall that we embed $\kk^*$ into $\kg^*$ as the annihilator of $\kp$.) By Propositions 3.13 and 3.14 in \cite{HM2}, we have $M_{\xi} = N_{\xi}$, including $\Spinc$-structures where relevant.
The $\Spinc$-quantisation $Q^{\Spinc}(M_{\xi}) = Q^{\Spinc}(N_{\xi}) $ of such a reduced space, for the values $\xi$ of $\mu$ we will need, is defined in Section 5.3 of \cite{PVI}.

Suppose the map 
 $\mu$ is \emph{$G$-proper}, in the sense that the inverse image of any cocompact set is cocompact.
 
\begin{theorem}[{$[Q^{\Spinc},R]=0$} for the discrete series] \label{thm [Q,R]=0}
In the setting of Proposition \ref{prop def Dirac ds Fredholm}, we have
\[
\ind_{\ds}(D_{fv}) = \bigoplus_{\lambda \in (\Lambda_+ + \rho_K)^{\reg}} m_{\lambda} \pi^{\ds}_{\lambda},
\]
where $(\Lambda_+ + \rho_K)^{\reg}\subset \Lambda_+ + \rho_K$ is the subset of regular elements, and
with $m_{\lambda} \in \Z$ given by
\begin{equation} \label{eq [Q,R]=0}
m_{\lambda} = \sum_{  \begin{array}{c} \vspace{-1.5mm} \scriptstyle{\sigma \in \cF(\kh) \text{ s.t.}}\\  \scriptstyle{\lambda - \rho_{\sigma} \in \sigma} \end{array}} Q^{\Spin^c}(M_{\lambda - \rho_{\sigma}}),
\end{equation}
where $(\kh) \in \cH_{\kk}$ is such that $([\kk^M, \kk^M]) = ([\kh, \kh])$. If no such $\kh$ exists, then $\ind_{\ds}(D_{fv}) = 0$, i.e.\ $m_{\lambda} = 0$ for all $\lambda$.
\end{theorem}

\begin{remark}
In \cite{HS}, where compact groups are considered,  it was not assumed that the set of  zeroes of $v$ is compact, only that $\mu$ is proper. The arguments in this paper actually show that Theorem \ref{thm [Q,R]=0} holds without Assumption \ref{ass taming}. We have not included this generalisation here, because the definition of the index is less straightforward in that case.
\end{remark}

\subsection{Invariant $\Spinc$-quantisation commutes with reduction}

Now let $G$ be any unimodular Lie group, such that $G/K$ is even-dimensional, and $G/K$ is equivariantly $\Spin$. In this subsection we suppose in addition that $G$ is reductive.
Consider any complete, $G$-invariant Riemannian metric on $TM$.
In Theorem 6.8 in \cite{HM2}, the invariant index of deformed $\Spinc$-Dirac operators was shown to satisfy an asymptotic version of the quantisation commutes with reduction principle. This result can be sharpened.
Consider the multiplicities $n_{\lambda} \in \Z$ in
\[
\pi_{\kp} = \sum_{\lambda \in \Lambda_+ + \rho_K} n_{\lambda} \pi^K_{\lambda}.
\]
\begin{theorem}[\QcwR\ for the trivial representation] \label{thm [Q,R0]=0}
If $f$ is $\psi$-admissible, then one has for any $G$-invariant, complete Riemannian metric on $M$,
\[
 \ind_{L^2_T}^G(D_{fv}) = \sum_{\lambda \in \Lambda_+ + \rho_K} n_{\lambda} m_{\lambda},
\]
with $m_{\lambda}$ as in \eqref{eq [Q,R]=0}.
\end{theorem}

In the cocompact case, every smooth section is transversally $L^2$. Hence
\beq{eq invar index cocpt}
\ind_{L^2_T}^G(D) = \dim\bigl(\ker(D^+)^G\bigr) - \dim\bigl( \ker (D^-)^G \bigr).
\eeq
(See Theorem 2.7 in \cite{MZ}.) By Lemma D.2 and Proposition D.3 in Bunke's appendix to \cite{MZ}, the above integer equals
\beq{eq Bunke}
I^G_*\bigl(\mu_M^G[D]\bigr),
\eeq
where $I^G\colon C^*G \to \C$ is defined by integrating functions over $G$, on the dense subalgebra $L^1(G)\subset C^*G$. Now we use the maximal group $C^*$-algebra rather than the reduced one. Furthermore, Assumption \ref{ass taming} holds automatically now. We do not need to assume that the Riemannian metric is a product metric, because the $K$-homology class
 of $D$ does not depend on the Riemannian metric. Therefore, we obtain the following $\Spinc$-version of Landsman's quantisation commutes with reduction conjecture \cite{HL, Landsman}. Compared to the main result in \cite{MZ}, this result applies in the more general $\Spinc$-setting, and also holds exactly, rather than asymptotically. (The assumptions that $G$ is reductive and unimodular, and $G/K$ is equivariantly $\Spin$ are not made in \cite{MZ}, however.)
 \begin{corollary}[$\Spinc$-Landsman conjecture]
 If $M/G$ is compact, then
for any $G$-invariant Riemannian metric on $M$, we have 
\[
I^G_*\bigl(\mu_M^G[D]\bigr) = \sum_{\lambda \in \Lambda_+ + \rho_K} n_{\lambda} m_{\lambda},
\]
with $m_{\lambda}$ as in \eqref{eq [Q,R]=0}.
 \end{corollary}
 
\subsection{Vanishing results on $\Spin$-manifolds} \label{sec AH}

In \cite{AH}, Atiyah and Hirzebruch showed that for compact, connected Lie groups (or equivalently, for circles) the equivariant index of a $\Spin$-Dirac operator on a compact $\Spin$ manifold is zero for nontrivial actions. This was generalised to cocompact actions in a $K$-theoretical setting in \cite{HMAH}.  There are also versions for the discrete series index and the invariant index.

 The action by $G$ on $M$ is called \emph{properly trivial} if every stabiliser group is maximal compact (i.e., is as large as it can be). Otherwise it is called properly nontrivial. 
 Let $G$ be as in the previous subsection, but without assuming it to be reductive.
\begin{theorem} \label{thm AH}
Suppose $M$ is $G$-equivariantly $\Spin$, and that the action is cocompact and properly nontrivial. If $D$ is the $\Spin$-Dirac operator, then
\begin{itemize}
\item For any $G$-invariant Riemannian metric on $M$,
$\ind_{L^2_T}^G(D)=0$.
\item If $G$ is connected and semisimple  with discrete series, and the Riemannian metric on $TM$ is a product metric, then
$
\ind_{\ds}(D) =0$.
\end{itemize}
\end{theorem}
In the cocompact case, we saw that $\ind_{L^2_T}^G(D)$ equals \eqref{eq invar index cocpt} and \eqref{eq Bunke}. Because of the latter equality, the first part of Theorem \ref{thm AH} also follows from the result in \cite{HMAH}.


\section{Decomposing the Dirac operator} \label{sec decomp Dirac}

The proofs of the results in Section \ref{sec prelim} are based on two induction results: Propositions \ref{prop ind ds} and \ref{prop [Q,Ind]=0 non-cocpt}. We prove these by decomposing the $\Spinc$ Dirac operator $D$ in an explicit way, as discussed in this section. 

In this section, unless stated otherwise, $G$ is any Lie group for which the adjoint action by $K$ on $\kp$ lifts to $\Spin(\kp)$. As before, suppose that $K<G$ is a maximal compact subgroup, and that $M$ and $G/K$ are even-dimensional. (Unimodularity of $G$ is not used in this section.) We assume that the Riemannian metric on $M$ is a product metric.

\subsection{Dirac operators on $N$ and $G/K$}

Let $P \to M$ be the $G$-equivariant $\Spinc$-structure used before. 
In Section 3.2 of \cite{HochsDS} and Section 3.2 of \cite{HM2}, an induction procedure of equivariant $\Spinc$-structures from $N$ to $M$ is described. Proposition 3.10 of \cite{HM2} is a $\Spinc$-slice theorem, which states that there is a $K$-equivariant $\Spinc$-structure $P_N$ on $N$, such that the induced $\Spinc$-structure on $M$ equals the $\Spinc$-structure originally given. The connection $\nabla^L$ on $L \to M$ restricts to a connection on the determinant line bundle $L_N = L|_N$ of this $\Spinc$-structure on $N$. This defines a $\Spinc$-momentum map $\mu_N\colon N \to \kk^*$, analogously to \eqref{eq def mu}. In Lemma 5.3 of \cite{HM2}, it was shown that the connection $\nabla^L$ can be chosen such that $\mu(N) \subset \kk^*$, and
\begin{equation} \label{eq mu mN}
\mu_N = \mu|_N.
\end{equation}


Since $M$ and $G/K$ are even-dimensional,  so is $N$. 
Let $\cS_N \to N$ be the spinor bundle associated to $P_N$. Let $\nabla^{\cS_N}$ be the spinor connection on $\cS_N$ defined by the Levi--Civita connection on $TN$ and the connection $\nabla^L|_N$ on $L_N = L|_N$.
 Let 
 $D_N$ be the associated $\Spinc$-Dirac operator on $\cS_N$. 

By Lemma 6.2 in \cite{HochsDS}, one has an equivariant vector bundle isomorphism
\begin{equation} \label{eq decomp S}
\cSM \cong G\times_K(\cS_N \otimes \pi_{\kp}).
\end{equation}
(Here we use graded tensor products.) At the level of smooth sections, we get
\begin{equation} \label{eq decomp Gamma SM}
\Gamma^{\infty}(\cSM) \cong \bigl(   \Gamma^{\infty}(G\times N, p_N^*\cS_N) \otimes \pi_{\kp} \bigr)^K,
\end{equation}
where $p_N\colon G \times N \to N$ is the natural projection map.

For $s \in \Gamma^{\infty}(\cS_N)$ and $\varphi \in C^{\infty}(G)\otimes \pi_{\kp}$, define $\sigma(s\otimes \varphi) \in \Gamma^{\infty}(p_N^*\cS)\otimes\pi_{\kp}$ by
\[
\bigl(\sigma(s\otimes \varphi)\bigr)(g, n) = s(n) \otimes \varphi(g),
\]
for $n \in N$ and $g \in G$.  Let $\varepsilon$ be the grading operator on $\cS_N$, equal to $\pm 1$ on $\cS_N^{\pm}$. 

\begin{proposition} \label{prop Dirac M N 1}
The map $\sigma$, together with \eqref{eq decomp Gamma SM}, defines a $G$-equivariant linear isomorphism
\[
 \bigl(\Gamma^{\infty}(\cS_N) \hat \otimes C^{\infty}(G)\otimes \pi_{\kp}\bigr)^K \cong \Gamma^{\infty}(\cSM),
\]
where $\hat \otimes$ denotes the tensor product completed in the Fr\'echet topology on $\Gamma^{\infty}(\cSM)$. Under this isomorphism, the Dirac operator $\DM$ corresponds to 
\begin{equation} \label{eq DM DN DGK}
D_N \otimes 1 + \varepsilon \otimes D_{G/K},
\end{equation}
where $D_{G/K}$ was defined in \eqref{eq def DGK}.
\end{proposition}

\begin{remark}
If $N$ is a point, then Proposition \ref{prop Dirac M N 1} reduces to Proposition 1.1 in \cite{Parthasarathy} (where one takes $V$ to be the trivial representation). If $G=K$ is compact, then one gets the trivial identity $D_N = D_N$. In Proposition 6.7 of  \cite{HochsDS}, it was shown that Proposition \ref{prop Dirac M N} holds at the level of principal symbols.
\end{remark}


\subsection{A reformulation}

We will in fact first prove a reformulation of Proposition \ref{prop Dirac M N 1}, and then deduce this proposition.

With respect to the decompositions \eqref{eq Abels TM} and \eqref{eq decomp S},
the Clifford action $c$ by $TM$ on $\cSM$ is given by
\begin{equation} \label{eq decomp c}
c[g, v, X] [g, s_N, y] = [g, c_N(v)s_N, y] + \ [g, \varepsilon s_N,  c_{\kp}(X) y].
\end{equation}
Here $g \in G$, $n \in N$, $v \in T_nN$, $X \in \kp$, $s_N \in (\cS_N)_n$, $y \in \pi_{\kp}$, and we used the Clifford actions $c_N\colon TN \to \End(\cS_N)$ and $c_{\kp}\colon \kp \to \End(\pi_{\kp})$. 
 Let $p_N^*D_N$ be the operator on $\Gamma^{\infty}(p_N^*\cS_N)$ given by
\[
(p_N^*D_N s)(g, n) = D_N \bigl(s(g, \relbar)\bigr)(n),
\]
for $s \in \Gamma^{\infty}(p_N^*\cS_N)$, $g \in G$ and $n \in N$. 

Fix an orthonormal basis $\{X_1, \ldots, X_{k}\}$ of $\kp$, and consider the operator
\begin{equation} \label{eq Dp}
D_{\kp} := \sum_{j=1}^{k} \cL_{X_j} \otimes c_{\kp}(X_j)
\end{equation}
on $\Gamma^{\infty}(\cSM)$, via \eqref{eq decomp Gamma SM}, where $\cL_{X_j}$ is the Lie derivative of sections of 
$p_N^*\cS_N$ with respect to $X_j$. 
Then we have the following decomposition of $\DM$.
\begin{proposition} \label{prop Dirac M N}
Under the identification \eqref{eq decomp Gamma SM}, one has
\[
\DM = p_N^*D_N + \varepsilon D_{\kp}, 
\]
restricted to $K$-invariant sections. 
\end{proposition}
This result implies Proposition \ref{prop Dirac M N 1}.

\medskip\noindent \emph{Proof of Proposition \ref{prop Dirac M N 1}.}
The map  $\sigma$ maps $K$-invariant sections to $K$-invariant sections, and its image is dense in $\bigl(\Gamma^{\infty}(p_N^*\cS)\otimes\pi_{\kp}\bigr)^K$. Furthermore, with notation as above,
\[
\begin{split}
\bigl( \sigma\bigl( D_Ns \otimes \varphi + \varepsilon  s\otimes D_{G/K}\varphi\bigr) \bigr)(g, n) &= (D_N s)(n) \otimes \varphi(g) + \varepsilon  s(n) \otimes (D_{G/K}\varphi)(g) \\
	&=\bigl( \bigl(p_N^*D_N + \varepsilon D_{\kp} \bigr)\sigma(s\otimes \varphi) \bigr)(g, n). 
\end{split}
\]
Proposition \ref{prop Dirac M N} states that $p_N^*D_N + \varepsilon D_{\kp}$, restricted to $K$-invariant sections, is the Dirac operator $\DM$.
\hfill $\square$

\medskip
It remains to prove Proposition \ref{prop Dirac M N}.

\subsection{The Levi--Civita connection}

To prove Proposition \ref{prop Dirac M N}, we start by decomposing the Levi--Civita connection on $TM$. Let $\nabla^N$ be the Levi--Civita connection on $TN$, and let $\nabla^{G/K}$ be the Levi--Civita connection on $T(G/K)$, for the Riemannian metric defined by the given inner product on $\kp$. 
Consider the projection map $p_{G/K}\colon G \times N \to G/K$.
Using
\[
p_{G/K}^* T(G/K) = p_{G/K}^* (G\times_K \kp) = G\times N \times \kp \to G\times N,
\]
we rewrite \eqref{eq Abels TM} as
\[
TM = \bigl(p_N^* TN \oplus  p_{G/K}^* T(G/K)\bigr)/K.
\]
In terms of the action map $p_M\colon G \times N \to M$,  this can be rephrased as
\[
p_M^*TM = p_N^* TN \oplus  p_{G/K}^* T(G/K).
\]
We find that the space $\XX(M)$ of vector fields on $M$ decomposes as 
\[
\XX(M) = \Gamma^{\infty}\bigl(G\times N, p_N^* TN \oplus  p_{G/K}^* T(G/K) \bigr)^K.
\]

Consider the connection
\[
\widetilde{\nabla}^M := p_N^*\nabla^N \oplus p_{G/K}^* \nabla^{G/K}
\]
on $p_N^* TN \oplus  p_{G/K}^* T(G/K)$.  Let $\nabla^M$ be the connection on $TM$ equal to the restriction of $\widetilde{\nabla}^M$ to $K$-invariant sections. In other words, $p_M^*\nabla^M = \widetilde{\nabla}^M$.
\begin{lemma} \label{lem LC decomp}
The connection $\nabla^M$ is the Levi--Civita connection on $TM$.
\end{lemma}
\begin{proof}
The fact that $\nabla^N$ and $\nabla^{G/K}$ preserve the Riemannian metrics on $N$ and $G/K$, respectively, implies that $\nabla^M$ preserves the Riemannian metric on $M$. Here we use the fact that the Riemannian metric on $TM$ is a product metric.

To show that $\nabla^M$ is torsion-free, we note that the torsion $\Tor^{\nabla^M}$ of $\nabla^M$ is a tensor, so it is enough to show it vanishes on a set of vector fields spanning $TM$. Therefore, we only need to show it vanishes on ($K$-invariant) vector fields of the forms $p_N^*v_N$ and $p_{G/K}^*v_{G/K}$, for $v_N \in \XX(N)$ and $v_{G/K} \in \XX(G/K) \cong ((C^{\infty}(G)\otimes  \kp))^K$. 

Now for  $v_N, w_N \in \XX(N)$, we have
\[
\nabla^M_{p_N^*v_N}(p_N^*w_N) = (p_N^*\nabla^N)_{p_N^*v_N}(p_N^*w_N) = p_N^*\bigl(  \nabla^N_{v_N} w_N\bigr).
\]
Hence, because $\nabla^N$ is torsion-free,
\[
\begin{split}
\nabla^M_{p_N^*v_N}(p_N^*w_N)  - \nabla^M_{p_N^*w_N} (p_N^*v_N) &= p_N^* \bigl(  \nabla^N_{v_N} w_N - \nabla^N_{w_N} v_N \bigr) \\
&= p_N^*[v_N, w_N] \\
&= [p_N^*v_N, p_N^*w_N].
\end{split}
\]
So
\[
\Tor^{\nabla^M}(p_N^*v_N, p_N^*w_N) = 0.
\]
One similarly shows that for all $v_{G/K}, w_{G/K} \in \XX(G/K)$,
\[
\Tor^{\nabla^M}(p_{G/K}^*v_{G/K}, p_{G/K}^*w_{G/K}) = 0.
\]
It therefore remains to show that
\begin{equation} \label{eq tor N G K}
\Tor^{\nabla^M}(p_N^*v_N, p_{G/K}^*v_{G/K}) = 0,
\end{equation}
with $v_N$ and $v_{G/K}$ as above.

Since each of the vector fields $p_N^*v_N$ and $p_{G/K}^*v_{G/K}$ is tangent to the directions the other vector field is constant in, their Lie bracket vanishes. Also, 
\[
\nabla^M_{p_N^*v_N} (p_{G/K}^*v_{G/K}) = (p_{G/K}^*\nabla^{G/K})_{p_N^*v_N} (p_{G/K}^*v_{G/K}) = 0,
\]
since the tangent map of $p_{G/K}$ is zero on the image of $p_N^*v_N$. Similarly, one has $\nabla^M_{p_{G/K}^*v_{G/K}} (p_{N}^*v_{N}) = 0$. So in particular,\footnote{Note that the Lie bracket of sections of $p_N^*TN$ and $p_{G/K}^*T(G/K)$, and the analogous expression to \eqref{eq tor nabla N G K}, do \emph{not} vanish in general. This is only the case for the pulled-back sections considered here, which is enough.}
\begin{equation} \label{eq tor nabla N G K}
\nabla^M_{p_N^*v_N} (p_{G/K}^*v_{G/K}) - \nabla^M_{p_{G/K}^*v_{G/K}} (p_{N}^*v_{N}) = 0.
\end{equation}
We conclude that \eqref{eq tor N G K} holds. So $\nabla^M$ is torsion-free, and hence indeed the Levi--Civita connection on $TM$.
\end{proof}

\subsection{Spinor connections}

The decomposition of the Levi--Civita connection in Lemma \ref{lem LC decomp} implies an analogous decomposition of the spinor connection $\nabla^{\cSM}$ on $\cSM$, associated to the connection $\nabla^{\LM}$ on $\LM \to M$. 
\begin{lemma} \label{lem decomp nabla SM}
In terms of the decomposition \eqref{eq decomp Gamma SM}, one has for all $X \in \kp$ and $v \in \XX(N)$,
\begin{equation} \label{eq decomp nabla SM}
\nabla^{\cSM}_{p_N^*v + {X^M}} = (p_N^*\nabla^{\cS_N})_{p_N^*v}  +  \cL_X, 
\end{equation}
where $\cL_X$ is the Lie derivative of sections of $\cS$ with respect to $X$.
\end{lemma}
\begin{proof}
Let $U \subset N$ be a $K$-invariant open subset such that
\[
\cS_N|_U = \cS_0^U \otimes (L_N|_U)^{1/2},
\]
where $\cS_0^U \to U$ is the spinor bundle for a local $\Spin$-structure on $U$. Then 
\[
\cSM|_{G\times_K U} = \bigl(G\times_K (\cS_0^U \otimes \pi_{\kp})\bigr) \otimes  \bigl( G\times_K (L_N|_U)\bigr)^{1/2}.
\]
We have
\[
\nabla^{\cSM}|_{G\times_K U} = \nabla^{\cS_0^{G\times_K U}} \otimes 1 + 1\otimes \nabla^{(\LM|_{G\times_K U})^{1/2}},
\]
where $\nabla^{\cS_0^{G\times_K U}}$ is the connection on the spinor bundle $\cS_0^{G\times_K U}\to G\times_K U$ induced by the Levi--Civita connection on $G\times_K U \hookrightarrow M$. 

First note that for all $s_{L_N} \in \Gamma^{\infty}(L_N)^K$, we have $p_N^*s_{L_N} \in \Gamma^{\infty}(p_N^*L_N)^K = \Gamma^{\infty}(\LM)$, and for all $X \in \kp$ and $v \in \XX(N)$,
\[
\nabla^{\LM}_{p_N^*v + {X^M}} (p_N^*s_{L_N}) = p_N^*\bigl( \nabla^{L_N}_{v} s_{L_N} \bigr) = (p_N^* \nabla^{L_N})_{p_N^*v} p_N^*s_{L_N}. 
\]
This follows from the definition of $\nabla^L$ in (22) in \cite{HochsDS}. 
Furthermore, let $\nabla^{\cS_0^U}$ be the connection on $\cS_0^U$ induced by $\nabla^N$, and let $\nabla^{\cS_0^{G/K}}$ be the connection on the spinor bundle $\cS_0^{G/K} = G\times_K\pi_{\kp} \to G/K$ induced by $\nabla^{G/K}$. Then Lemma \ref{lem LC decomp} implies that 
one has
\[
\nabla^{\cS_0^{G\times_K U}} = p_N^*\nabla^{\cS_0^U}  + p_{G/K}^*\nabla^{\cS_0^{G/K}},
\]
restricted to $K$-invariant sections.
The connection $\nabla^{\cS_0^{G/K}}$ on $\Gamma^{\infty}(\cS_0^{G/K}) = \bigl(C^{\infty}(G) \otimes \pi_{\kp} \bigr)^K$ is simply given by
\[
\nabla^{\cS_0^{G/K}}_{X^{G/K}} s_{G/K} = \cL_X(s_{G/K}),
\]
for $X \in \kp$ and $s_{G/K} \in \bigl(C^{\infty}(G) \otimes \pi_{\kp} \bigr)^K$. (As noted on page 7 of \cite{Parthasarathy}, the connection $\nabla^{G/K}$ is induced by the canonical connection on the principal fibre bundle $G \to G/K$.)

Since both sides of \eqref{eq decomp nabla SM} satisfy the Leibniz rule, it is enough to check this equality on a set of sections spanning $\cSM|_U$. Hence it is enough to consider a section
\[
s:= p_N^*(s_N \otimes s_{L_N}) \otimes p_{G/K}^*s_{G/K} \in \Gamma^{\infty}(\cSM|_{G\times_K U}),
\]
for
\[
\begin{split}
s_N &\in \Gamma^{\infty}(\cS_0^U)^K; \\
s_{G/K}&\in  \bigl(C^{\infty}(G) \otimes \pi_{\kp} \bigr)^K;\\
s_{L_N} &\in\Gamma^{\infty} \bigl(L_N|_U^{1/2}\bigr)^K.
\end{split}
\]
For such a section, and for all $X \in \kp$ and $v \in TU$, the preceding arguments allow us to compute
\[
\begin{split}
\nabla^{\cSM}_{p_N^*v + {X^M}}s &= \nabla^{\cS_0^{G\times_K U}}_{p_N^*v + {X^M}} \bigl(  p_N^*s_N \otimes p_{G/K}^*s_{G/K}\bigr)\otimes p_N^*s_{L_N} \\
		& \qquad + 
		p_N^*s_N \otimes p_{G/K}^*s_{G/K} \otimes \nabla^{\LM|_{G\times_K U}^{1/2}}_{p_N^*v + {X^M}} p_N^*s_{L_N} \\
		&= 
		(p_N^*\nabla^{\cS_0^U})_{p_N^*v}(p_N^*s_N) \otimes p_{G/K}^*s_{G/K}  \otimes p_N^*s_{L_N}   \\
		& \qquad + 
		p_N^*s_N \otimes \cL_X(p_{G/K}^*s_{G/K})  \otimes p_N^*s_{L_N} \\
		& \qquad + p_N^*s_N \otimes p_{G/K}^*s_{G/K} \otimes p_N^*\bigl(\nabla^{L_N|_{U}^{1/2}}_v s_{L_N}\bigr) \\
		&= 
		(p_N^*\nabla^{\cS_N})_{p_N^*v}\bigl(p_N^*(s_N \otimes s_{L_N}) \bigr) \otimes p_{G/K}^*s_{G/K}  \\
		& \qquad + 
		p_N^*(s_N \otimes s_{L_N}) \otimes \cL_X(p_{G/K}^*s_{G/K}) \\
			&= 
		\bigl( (p_N^*\nabla^{\cS_N})_{p_N^*v} + \cL_X\bigr)s,
\end{split}
\]
since $(p_N^*\nabla^{\cS_N})_{p_N^*v}$ vanishes on sections pulled back from $G/K$, while $X$ vanishes on sections pulled back from $N$.
\end{proof}

\subsection{Proof of Proposition \ref{prop Dirac M N}}

Using Lemma \ref{lem decomp nabla SM}, we can prove Proposition \ref{prop Dirac M N}. One ingredient of the proof is the following expression for the operator $p_N^*D_N$.
\begin{lemma} \label{lem pNDN}
If $\{e_1, \ldots, e_l\}$ is a local orthonormal frame for $TN$, then locally,
\begin{equation} \label{eq pNDN}
p_N^*D_N = \sum_{s=1}^l c(p_N^*e_s)(p_N^*\nabla^{\cS^N})_{p_N^*e_s}.
\end{equation}
\end{lemma}
\begin{proof}
Note that any section of $\Gamma^{\infty}(p_N^*\cS_N)$ is a sum of sections of the form $\varphi p_N^*s_N$, for $\varphi \in C^{\infty}(G \times N)$ and $s_N \in \Gamma^{\infty}(\cS_N)$. On such a section, one has 
\begin{equation}
(p_N^*\nabla^{\cS^N})_{p_N^*e_s} (\varphi p_N^*s_N) 
	=  \varphi  p_N^*\bigl(\nabla^{\cS_N}_{p_N^* e_s} s_N \bigr) + {(p_N^* e_s)} (\varphi) p_N^*s_N. \label{eq nab SM 2}
\end{equation}
At a point $(g, n) \in G\times N$, one has
\[
(p_N^* e_s)(\varphi)(g, n) =  e_s\bigl( \varphi(g, \relbar) \bigr)(n). 
\]
Therefore, at such a point, we find that \eqref{eq nab SM 2} equals
\[
\left(   \nabla^{\cS_N}_{e_s} \varphi(g, \relbar) s_N \right)(n). 
\]
We conclude that, at $(g, n)$,  the right hand side of \eqref{eq pNDN} applied to $\varphi p_N^*s_N$ yields
\[
\begin{split}
\left(\sum_{s=1}^l c(p_N^* e_s)\nabla^{\cS_N}_{p_N^* e_s} (\varphi p_N^*s_N) \right)(g, n)&= \sum_{s=1}^l  \left(c(e_s)\left(   \nabla^{\cS_N}_{e_s} \varphi(g, \relbar) s_N \right)\right)(n) \\
	&= \bigl((p_N^*D_N)(\varphi p_N^*s_N)\bigr) (g, n).
\end{split}
\]
\end{proof}

\medskip
\noindent
\emph{Proof of Proposition \ref{prop Dirac M N}.}
Let $\{X_1, \ldots, X_k\}$ be an orthonormal basis of $\kp$, and let $\{e_1, \ldots, e_l\}$ be a local orthonormal frame for $TN$. Then, because the Riemannian metric on $TM$ is a product metric,
\begin{equation} \label{eq DM}
\DM = \sum_{r=1}^k c(X_r)\nabla^{\cSM}_{X_r^M} + \sum_{s=1}^l c(p_N^*e_s)\nabla^{\cSM}_{p_N^*e_s}.
\end{equation}
Note that for each $r$ and $s$, $c(X_r)$ acts on $\pi_{\kp}$, and $c(p_N^*e_s)$ acts on $\cS_N$ in $\cSM = G\times_K(\cS_N \otimes \pi_{\kp})$, via \eqref{eq decomp c}. 

By Lemma \ref{lem decomp nabla SM} and \eqref{eq decomp c}, the first term on the right hand side of  \eqref{eq DM} equals
\[
\sum_{r=1}^k c(X_r)\cL_{X_r} = \varepsilon D_{\kp}.
\]
The same lemma implies that the second term 
 equals
\[
 \sum_{s=1}^l c(p_N^*e_s)(p_N^*\nabla^{\cS_N})_{p_N^*e_s},
\]
which by Lemma \ref{lem pNDN} equals $p_N^*D_N$.
\hfill $\square$


\section{Induction}\label{sec induction}

We prove  two induction results, Propositions \ref{prop ind ds} and \ref{prop [Q,Ind]=0 non-cocpt}, by using Proposition \ref{prop Dirac M N 1} and keeping track of the $L^2$-norms on the various spaces involved. These induction results are then used to prove the results in Section \ref{sec prelim}.
To compare $L^2$-norms, we use a relation between the Riemannian densities on $M$, $N$ and $G$.

In this section, initially $G$ can be any Lie group, with a fixed Haar measure $dg$ and maximal compact subgroup $K$. We still assume that the Riemannian metric on $M$ is a product metric.




\subsection{Densities}

Recall that by assumption,  the Riemannian metric on $M = G\times_K N$ is induced by the given inner product on $\kp$ and a $K$-invariant Riemannian metric on $N$. Let  $dm$  and $dn$ be the densities on $M$ and $N$ defined by these Riemannian metrics. 
 Let $dk$ be the Haar measure on $dk$ giving $K$ unit volume.
We will prove and use the fact that $dm$ equals the measure $d[g, n]$ on $G\times_K N$ induced by the product measure $dg \times dn$ on $G\times N$, via the equality
\[
\int_{G\times N} \varphi(g, n) \, dg \, dn = \int_{G\times_K N} \int_K \varphi(k\cdot \tau[g, n]) \, dk\, d[g,n]
\]
for any $\varphi \in C_c(G\times N)$ and any Borel section $\tau\colon G\times_K N \to G \times N$. (See e.g. \cite{Bourbaki}, Chapter 7, Section 2, Proposition 4b.)

\begin{lemma} \label{lem dm}
Under the diffeomorphism $G\times_K N = M$ defined by the action,
and for a suitable scaling of the Haar measure $dg$, the measure $d[g, n]$ corresponds to $dm$. 
\end{lemma}
\begin{proof}
Consider the non-equivariant diffeomorphisms
\[
\begin{split}
\Psi_{M}\colon &\kp \times N \to M; \\
\Psi_{G \times N}\colon &\kp \times K \times N \to G \times N,
\end{split}
\]
defined by
\[
\begin{split}
\Psi_{M}(X, n) &= \exp(X)n; \\
\Psi_{G \times N}(X, k, n) &= (\exp(X)k^{-1}, kn),
\end{split}
\]
for $X \in \kp$, $n \in N$ and $k \in K$.

Let $dX$ be the Riemannian density on $\kp$. 
Then, since $\Psi_M$ is an isometry, 
\begin{equation} \label{eq Psi M}
\Psi_M^*dm = dX \otimes dn.
\end{equation}
Now let the Haar measure $dg$ be given by the $G$-invariant Riemannian metric induced by the inner product on $\kg$. Let $dk$ be the Haar measure on $K$ defined in the same way. By rescaling the inner product on $\kg$, we can make sure that $dk$ gives $K$ unit volume. By Lemma \ref{lem Psi G N} below, we have
\begin{equation} \label{eq Psi G N}
\Psi_{G \times N}^*(dg \otimes dn) = dX \otimes dk \otimes dn.
\end{equation}

The equalities \eqref{eq Psi M} and \eqref{eq Psi G N} imply that for all $\varphi \in C_c(M)$,
\[
\begin{split}
\int_M \varphi(m) dm &= \int_{\kp \times N} \varphi(\exp(X)n) dX \otimes dn \\
	&= \int_{\kp \times K \times N} \varphi(\exp(X)n) dX \otimes dk \otimes dn \\
	&= \int_{G\times N} \varphi(gn) dg\otimes dn \\
	&= \int_{G\times_K N} \varphi(gn) d[g, n],
\end{split}
\]
where we used the fact that the map $(g, n) \mapsto gn$ is invariant under the $K$-action given by $k\cdot (g, n) = (gk^{-1}, kn)$.
\end{proof}

\begin{lemma} \label{lem Psi G N}
In the notation of the proof of Lemma \ref{lem dm}, we have
\[
\Psi_{G \times N}^*(dg \otimes dn) = dX \otimes dk \otimes dn.
\]
\end{lemma}
\begin{proof}
One can compute that for all $X, Y \in \kp$, $Z \in \kk$, $k \in K$, $n \in N$ and $v \in T_nN$,
\[
T_{(X, k, n)} \Psi_{G\times N}(Y, T_el_k(Z), v) = \bigl( T_el_{\exp(X)k^{-1}}\bigl( \Ad(k)(Y + Z)\bigr), T_n k\bigl(\alpha_n(Z) + v \bigr) \bigr).
\]
Here the letter $l$ denotes left multiplication, and for $m \in M$, the map $\alpha_m\colon \kg \to T_mM$ is given by the infinitesimal action. Now the maps $T_el_{\exp(X)k^{-1}}$, $\Ad(k)$ and $T_nk$ preserve the Riemannian metrics on $TG$ and $TN$. So
\[
T_{(X, k, n)} \Psi_{G\times N} = B\circ A,
\]
where
\[
A\colon T_{(X, k, n)}(\kp \times K \times N) \to T_{(\exp(X)k^{-1}, kn)}(G\times N),
\]
given by
\[
A(Y, T_el_k(Z), v) = \bigl( T_el_{\exp(X)k^{-1}}\bigl( \Ad(k)(Y + Z)\bigr), T_nk( v) \bigr) 
\]
is an isometry, and the automorphism $B$ of
\[
T_{(\exp(X)k^{-1}, kn)}(G\times N) \cong \kp\oplus \kk \oplus T_{kn}N
\]
is given by the matrix
\[
\mat(B) = \left( \begin{array}{ccc}I_{\kp} & 0 & 0 \\ 0 & I_{\kk} & 0 \\ 0 & T_nk \circ \alpha_n & I_{T_{kn}N}\end{array}\right),
\]
where $I_{\kk}$, $I_{\kp}$ and $I_{T_{kn}N}$ are the identity maps on the respective spaces, 
so that $B$ has determinant one.

Since the map $A$ is an isometry, it relates the Riemannian density $dX \otimes dk \otimes dn$ on $\kp \times K \times N$ to the Riemannian density $dg\otimes dn$ on $G\times N$, at the point $(X, k, n)$. Since the map $B$ has unit determinant, it does not change densities, so the claim follows.
\end{proof}


\begin{lemma} \label{lem L2 S}
In the notation of Proposition \ref{prop Dirac M N 1}, we have
\[
\|\sigma(s\otimes \varphi)\|_{L^2(\cS)} = \|s\|_{L^2(\cS_N)} \|\varphi\|_{L^2(G)\otimes  \pi_{\kp}},
\]
for all $s \in \Gamma^{\infty}_c(\cS_N)$ and $\varphi \in C^{\infty}_c(G)\otimes \pi_{\kp}$ such that $s\otimes \varphi$ is $K$-invariant.
\end{lemma}
\begin{proof}
By Lemma \ref{lem dm} and $K$-invariance of $s\otimes \varphi$ and of the norm on $\cS$, and implicitly using a Borel section $G\times_K N \to G\times N$, one has
\[
\begin{split}
\|\sigma(s\otimes \varphi)\|_{L^2(\cS)}^2 &= \int_{G\times_K N} \| s(n) \otimes \varphi(g) \|_{\cS}^2\,  d[g, n] \\
&= \int_{G\times_K N} \int_K \| s(n) \otimes \varphi(g) \|_{\cS}^2 \, dk\, d[g, n] \\
&= \int_{G\times N} \| s(n) \|_{\cS_N}^2 \| \varphi(g) \|_{\pi_{\kp}}^2 \, dg \, dn \\
&= \|s\|_{L^2(\cS_N)}^2 \|\varphi\|_{L^2(G)\otimes  \pi_{\kp}}^2.
\end{split}
\]
\end{proof}

\subsection{Deformed Dirac operators}

Now suppose $G/K$ is even-dimensional and equivariantly $\Spin$.
Consider a real-valued function $f\in C^{\infty}(M)^G = C^{\infty}(N)^K$,
and  the deformed Dirac operator
\[
D_{fv_N} := D_N + ic_N(fv_N).
\]
\begin{proposition} \label{prop ker deformed Dirac}
The map $\sigma$ defines a $G$-equivariant, graded, unitary isomorphism
\[
\ker_{L^2} D_{fv}\cong \bigl( \ker_{L^2} (D_{fv_N}) \otimes \ker_{L^2} (D_{G/K}) \bigr)^K
\]
(Here the tensor product is completed in the $L^2$-inner product.)
\end{proposition}
\begin{proof}
Since the algebraic tensor product
\[
\Gamma^{\infty}_c(\cS_N) \otimes C^{\infty}_c(G)\otimes \pi_{\kp}
\]
is dense in 
\[
L^2(\cS_N) \otimes L^2(G)\otimes  \pi_{\kp},
\]
Proposition \ref{prop Dirac M N 1} and Lemma \ref{lem L2 S} imply that $\sigma$ induces a unitary isomorphism
\[
L^2(\cS) \cong \bigl( L^2(\cS_N) \otimes L^2(G)\otimes  \pi_{\kp} \bigr)^K.
\]
By Proposition \ref{prop Dirac M N 1}, this isomorphism intertwines the operators $D_M$ and $D_{N} \otimes 1 +  \varepsilon\otimes D_{G/K}$. Since it also intertwines $c(v)$ and $c(v_N) \otimes 1$, it intertwines $D_{fv}$ and $D_{fv_N} \otimes 1 +  \varepsilon\otimes D_{G/K}$.

As in the proof of Theorem 3.5 in \cite{Atiyah}, the presence of the grading operator $\varepsilon$ in \eqref{eq DM DN DGK} implies that
\[
\bigl( D_{fv_N} \otimes 1 + \varepsilon\otimes D_{G/K} \bigr)^2 = D_{fv_N}^2 \otimes 1 + 1\otimes D_{G/K}^2.
\]
Since the operators $D_{fv_N}$ and $D_{G/K}$ are 
symmetric, we find that
\[
\ker_{L^2} D_{fv} \cong \bigl( \ker_{L^2} (D_{fv_N}) \otimes \ker_{L^2} (D_{G/K}) \bigr)^K.
\]
Since the isomorphism is compatible with the gradings, the claim follows. 
\end{proof}

Proposition \ref{prop ker deformed Dirac} holds at the level of kernels. To prove the results in Section \ref{sec prelim}, we only need the corresponding weaker result about indices. 
Suppose $G$ is semisimple with discrete series.
Consider the Dirac induction map \eqref{eq def Dirac ind}. Note that for $(\lambda \in \Lambda_+ + \rho_K)^{\reg}$, we have
\[
\widehat{\DInd}_K^G(\pi^K_{\lambda}) = \pi^{\ds}_{\lambda}.
\] 
The following induction result for indices follows directly from Proposition \ref{prop ker deformed Dirac} and Theorem \ref{thm AS Parth}.
\begin{proposition}\label{prop ind ds}
In the setting of Proposition \ref{prop def Dirac ds Fredholm}, we have
\[
\ind_{\ds}D_{fv} = \widehat{\DInd}_K^G (\ind_K D_{fv_N}).
\]
\end{proposition}

\subsection{Invariant parts}

In this subsection, $G$ is unimodular, and $G/K$ is even-dimensional and equivariantly $\Spin$.
We now consider $G$-invariant, transversally $L^2$ sections of $\cSM$, to prove Proposition \ref{prop [Q,Ind]=0 non-cocpt}. 
%
\begin{lemma}
Restriction to $N$ is a  linear isomorphism
\[
\Gamma^{\infty}(\cSM)^G \xrightarrow{\cong} \bigl(\Gamma^{\infty}(\cS_N ) \otimes \pi_{\kp}\bigr)^K.
\] 
\end{lemma}
\begin{proof}
Note that for all $s \in \Gamma^{\infty}(\cSM)^G$ and $n \in N$, we have $s(n) \in \cSM_{n}\cong (\cS_N)_n \otimes \pi_{\kp}$. Every $K$-invariant section in $\bigl(\Gamma^{\infty}(\cS_N ) \otimes \pi_{\kp}\bigr)^K$ has a unique $G$-invariant extension to a section in  $\Gamma^{\infty}(\cSM)^G$. This is the inverse to the restriction map.
\end{proof}

Fix $s \in \Gamma^{\infty}(\cSM)^G$. 
Let $h_G \in C^{\infty}(G)^K$ be such that
\[
\int_G h_G(g)^2 \, dg = 1
\]
for the Haar measure $dg$ as in Lemma \ref{lem dm}. Here the superscript $K$ denotes invariance under right multiplication by $K$. Define the cutoff function $h \in C^{\infty}(M)$ by
\[
h(gn) = h_G(g),
\]
for $g \in G$ and $n \in N$. 

The characterisation of the density $dm$ in Lemma \ref{lem dm} allows us to relate transversally $L^2$ sections on $M$ to $L^2$-sections on $N$.
\begin{lemma}\label{lem L2T}
We have
\[
\| hs \|_{L^2(\cSM)} = \|s|_N\|_{L^2(\cS_N) \otimes \pi_{\kp}}.
\]
\end{lemma}
\begin{proof}
By Lemma \ref{lem dm}, we have
\[
\begin{split}
\| hs \|_{L^2(\cSM)}^2 &= \int_M h(m)^2 \|s(m)\|^2_{\cS} \, dm \\
	&= \int_{G\times_K N} h_G(g)^2 \|g^{-1}(s(n))\|_{\cSM}^2 \, d[g, n] \\
	&=  \int_{G\times N} h_G(g)^2 \|s(n)\|_{\cS_N \otimes \pi_{\kp}}^2 \, dg\,  dn \\
	&= \|s|_N\|_{L^2(\cS_N) \otimes \pi_{\kp}}^2,
\end{split}
\]
where we have used $G$-invariance of the metric $\|\cdot \|_{\cSM}$ and $K$-invariance of $s|_N$.
\end{proof}

Unimodularity of $G$ implies that the definition of the space $L^2_T(\cSM)^G$ is independent of the cutoff function chosen. Lemma \ref{lem L2T} has the following consequence.
\begin{lemma} \label{lem L2T M L2 N}
Restriction to $N$ is a graded unitary isomorphism
\[
L^2_T(\cSM)^G \cong \bigl( L^2(\cS_N) \otimes \pi_{\kp}\bigr)^K.
\]
\end{lemma}


In Proposition \ref{prop Dirac M N 1}, the operator $D_{G/K}$ is zero on $G$-invariant sections. It therefore has the following consequence.
\begin{lemma} \label{lem Dirac M N invar}
One has
\[
(\DM s)|_N =  (D_N \otimes 1_{\pi_{\kp}}) (s|_N).
\]
\end{lemma}


Because of \eqref{eq mu mN}, we have
$v_N = v|_N$.
Therefore,
\begin{equation} \label{eq cvM cvN}
\bigl(c(v)s\bigr)|_N = c_N(v_N)s|_N
\end{equation}
Lemmas \ref{lem L2T M L2 N} and  \ref{lem Dirac M N invar}, together with \eqref{eq cvM cvN}, yield the following conclusion.
\begin{proposition} \label{prop L2T ker}
We have a graded linear isomorphism
\[
\ker_{L^2_T}^G (D_{f\vM}) \cong \bigl(\ker_{L^2}(D_{fv_N})\otimes \pi_{\kp} \bigr)^K.
\]
\end{proposition}

Using Proposition \ref{prop L2T ker} and the fact that $\pi_{\kp}\cong \pi_{\kp}^*$, we obtain the desired induction result.
\begin{proposition}\label{prop [Q,Ind]=0 non-cocpt}
We have
\[
\ind_{L^2_T}^G(D_{fv}) =   \bigl[ \ind_{K}(D_{fv_N}): \pi_{\kp} \bigr] \quad \in \Z.
\]
\end{proposition}
We have so far assumed that  the Riemannian metric on $M$ is a product metric in this section. However,  the invariant index is independent of the (complete, $G$-invariant) Riemannian metric by cobordism invariance, Theorem 3.6 in \cite{Braverman}. Furthermore, any Riemannian manifold with a complete, $G$-invariant Riemannian metric has a complete product metric by Lemma 3.12 in \cite{HS3}. Therefore, Proposition \ref{prop [Q,Ind]=0 non-cocpt} holds for \emph{any} complete, $G$-invariant Riemannian metric on $TM$.


\subsection{Proofs of the results} \label{sec proofs}

Let us prove the results in Section \ref{sec prelim}. Proposition \ref{prop def Dirac ds Fredholm} follows directly from Theorem \ref{thm AS Parth}, Proposition \ref{prop ker deformed Dirac} and well-definedness of the index \eqref{eq index K} for compact groups.

Corollary 3.8 in \cite{HS3} implies that
\beq{eq ind Khom}
\ind_G(D_{fv}) = \ind_K(D_{fv_N})\otimes \pi_{\kp} \quad \in \hat R(K).
\eeq
Hence Proposition \ref{prop invar index} follows from Proposition \ref{prop [Q,Ind]=0 non-cocpt}:
\[
\ind^G_{L^2_T}(D_{fv}) =\dim \bigl(\ind_K D_{fv_N}\otimes \pi_{\kp}\bigr)^K = \dim\bigl(\ind_G(D_{fv})^K \bigr).
\]
In the same way,
Proposition \ref{prop G ds ind} follows from \eqref{eq ind Khom} and
 Proposition \ref{prop ind ds}.

To prove
Proposition \ref{prop ass map}, we note that by (5.3) in \cite{HochsPS},
\[
\DInd_K^G(\pi^K_{\lambda}) = (-1)^{\dim (G/K)/2}j(\pi^{\ds}_{\lambda})\in K_0(C^*_rG).
\]
Here $\DInd_K^G$ is the $K$-theoretical Dirac induction map from the Connes--Kasparov conjecture (see Conjecture 4.20 in \cite{BCH}).
Therefore, by the induction result for the analytic assembly map, Theorem 5.8 in \cite{HM2}, we have
\[
\begin{split}
p_{\ds}(\mu_M^G[D]) &= p_{\ds} \circ \DInd_K^G (\ind_K D_N)\\
	&=(-1)^{\dim (G/K)/2} \bigoplus_{\lambda \in (\Lambda_+ + \rho_K)^{\reg}} [\ind_K(D): \pi^K_{\lambda}] j(\pi^{\ds}_{\lambda}) \\
	&= (-1)^{\dim (G/K)/2}j\bigl(\widehat{\DInd}_K^G (\ind_K(D_N))\bigr) \\
	&=  (-1)^{\dim (G/K)/2} j(\ind_{\ds} (D)),
\end{split}
\]
by Proposition \ref{prop ind ds}.

Theorem \ref{thm [Q,R]=0} follows from the corresponding result for compact groups, Theorem 3.10 in \cite{HS}, via Proposition \ref{prop ind ds}. Here we also use the fact that $M_{\xi} = N_{\xi}$ for all $\xi \in \kk^*$, if $G$ is reductive (see Proposition 3.13 in \cite{HM2}).
In a similar way, we can use Proposition \ref{prop [Q,Ind]=0 non-cocpt} to deduce Theorem \ref{thm [Q,R0]=0} from Theorem 3.9 in \cite{HS}. In that case, we do not need to assume that the Riemannian metric on $TM$ is a product metric.

Finally, Theorem \ref{thm AH} follows from Atiyah and Hirzebruch's result in \cite{AH} via Propositions \ref{prop ind ds} and  \ref{prop [Q,Ind]=0 non-cocpt}. Indeed, the condition that the action is properly nontrivial is equivalent to the action by $K$ on $N$ being nontrivial, see Lemma 9 in \cite{HMAH}. Also, the $\Spinc$-structure $P_N$ is now a $\Spin$-structure, see Lemma 10 in \cite{HMAH}.


%

\end{document}